\documentclass[10pt]{amsart}
\usepackage[margin=1.2in,marginparsep=0.1in,marginparwidth=1in]{geometry}
\usepackage{amssymb,amsmath,amsthm,amstext,amscd,latexsym,graphics,graphicx,bbm,caption}
\usepackage[usenames,dvipsnames,svgnames,table]{xcolor}
\usepackage[plainpages=false,colorlinks=true, pagebackref]{hyperref}
\usepackage{tikz-cd}
\usepackage{tikz}
\usetikzlibrary{positioning}
\usepackage{mathrsfs}

\tikzset{>=stealth}
\hypersetup{citecolor=Sepia,linkcolor=blue, urlcolor=blue}

\usepackage{cite}   

\makeatletter
\def\@tocline#1#2#3#4#5#6#7{\relax
  \ifnum #1>\c@tocdepth 
  \else
    \par \addpenalty\@secpenalty\addvspace{#2}%
    \begingroup \hyphenpenalty\@M
    \@ifempty{#4}{%
      \@tempdima\csname r@tocindent\number#1\endcsname\relax
    }{%
      \@tempdima#4\relax
    }%
    \parindent\z@ \leftskip#3\relax \advance\leftskip\@tempdima\relax
    \rightskip\@pnumwidth plus4em \parfillskip-\@pnumwidth
    #5\leavevmode\hskip-\@tempdima
      \ifcase #1
       \or\or \hskip 2em \or \hskip 2em \else \hskip 3em \fi%
      #6\nobreak\relax
    \dotfill\hbox to\@pnumwidth{\@tocpagenum{#7}}\par
    \nobreak
    \endgroup
  \fi}
\makeatother

\newtheorem{intro-thm}{Theorem}[]
\theoremstyle{plain}
\newtheorem{thm}{Theorem}[section]
\newtheorem{theorem}[thm]{Theorem}
\newtheorem{q}[thm]{Question}
\newtheorem{lemma}[thm]{Lemma}
\newtheorem{corollary}[thm]{Corollary}
\newtheorem{proposition}[thm]{Proposition}

\theoremstyle{definition}
\newtheorem{remark}[thm]{Remark}

\newtheorem{definition}[thm]{Definition}
\newtheorem{example}[thm]{Example}

\newcommand{\inj}{\hookrightarrow}
\newcommand{\tensor}{\otimes}

\newcommand{\Spec}{{\rm Spec \,}}

\renewcommand{\tilde}{\widetilde}

\newcommand{\sE}{{\mathcal E}}
\newcommand{\sF}{{\mathcal F}}
\newcommand{\sG}{{\mathcal G}}

\newcommand{\sL}{{\mathcal L}}

\newcommand{\sO}{{\mathcal O}}

\newcommand{\sX}{{\mathcal X}}
\newcommand{\sY}{{\mathcal Y}}

\newcommand{\G}{{\mathbb G}}

\newcommand{\N}{{\mathbb N}}

\renewcommand{\P}{{\mathbb P}}

\newcommand{\Z}{{\mathbb Z}}

\usepackage{verbatim}

\input{xy}
\xyoption{all}
\begin{document}

\title{Quasi-Affineness and the 1-Resolution Property}

\author{Neeraj Deshmukh}
\address{ Department of Mathematics, Indian Institute of Science Education and Research (IISER) Pune, Dr. Homi Bhabha road, Pashan, Pune 411008 India}

\email{neeraj.deshmukh@students.iiserpune.ac.in}

\author{Amit Hogadi}
\address{Department of Mathematics, Indian Institute of Science Education and Research (IISER) Pune, Dr. Homi Bhabha road, Pashan, Pune 411008 India}
\email{amit@iiserpune.ac.in}

\author{Siddharth Mathur}
\address{Department of Mathematics\\University of Arizona\\617 N. Santa Rita Ave.%
\\P.O. Box 210089\\ Tucson, AZ 85721-0089 USA\\%
}
\email{smathur@math.arizona.edu} 



\date{}

\begin{abstract}

We prove that, under mild hypothesis, every normal algebraic space which satisfies the $1$-resolution property is quasi-affine. More generally, we show that for algebraic stacks satisfying similar hypotheses, the 1-resolution property guarantees the existence of a finite flat cover by a quasi-affine scheme.
\end{abstract}

\maketitle

\section{Introduction}


An algebraic stack is said to have the resolution property if every quasi-coherent sheaf of finite-type is the quotient of a vector bundle. Although quasi-projective schemes always enjoy this property, it is not well-understood in general. The Totaro--Gross Theorem establishes that quasi-compact and quasi-separated (qcqs) algebraic stacks with affine stabilizers and the resolution property are quotients of quasi-affine schemes by $\text{GL}_n$ (see \cite{Tot}, \cite{GrossRes}). 

In this paper, we will consider a special case of the resolution property where every quasi-coherent sheaf of finite-type is generated by a single vector bundle in the following sense:

\begin{definition}

An algebraic stack $\sX$ is said to have the {\em 1-resolution property} if it admits a vector bundle $V$ such that every quasi-coherent sheaf of finite-type on $\sX$ is a quotient of $V^{\oplus n}$ for some natural number $n$. We will say that such a 
$V$ is \textit{special}.
\end{definition}

Hall and Rydh considered this property for algebraic stacks and posed the following question:

\begin{q}\label{q}(see \cite[7.4]{HR})
Does every algebraic stack with the 1-resolution property admit a finite flat covering by a quasi-affine scheme? Moreover, is every algebraic space with the 1-resolution property quasi-affine?
\end{q}

Our main result addresses this question and gives a positive answer under moderate hypotheses.

\begin{theorem}\label{T}
Let $\sX$ be a Noetherian, quasi-excellent (see \cite[7.2]{Tannaka}) and normal algebraic stack whose stabilizers at closed points are affine. Then we have:
\begin{enumerate}
\item If $\sX$ is representable by an algebraic space $X$, then $X$ has the 1-resolution property if and only if it is quasi-affine.
\item $\sX$ has the 1-resolution property if and only if there exists a finite flat covering $U\rightarrow \sX$ with $U$ a quasi-affine scheme.
\end{enumerate}
\end{theorem}

Note that by \cite[Lem C.1]{Rydh11} any algebraic space which admits a finite flat covering by a quasi-affine scheme is automatically quasi-affine. Thus, although statement \ref{T}(1) can be treated as a special case of statement \ref{T}(2), we state them separately. Indeed, the strategy we adopt is to first establish \ref{T}(1), and then use it to prove \ref{T}(2).

The `if' direction is just descent for the 1-resolution property along finite faithfully flat maps (see \cite[Prop. 2.13]{GrossRes}) combined with the fact that on a quasi-affine scheme $U$ every coherent sheaf is globally generated (i.e, $\sO_U$ is special). So, this paper is concerned with proving the `only if' direction.

Recall the following theorem of Totaro and Gross.

\begin{theorem} \label{Grossaffine} \cite[Prop. 1.3]{Tot}, \cite[Theorem 1.1]{GrossRes}
Let $\sX$ be a qcqs algebraic stack whose stabilizer groups at closed points are affine. If $\sX$ has the resolution property, then the diagonal morphism $\sX\rightarrow \sX\times_{\Z}\sX$ is affine.
\end{theorem}
It is because of this theorem that all schemes and stacks considered in this paper will automatically have affine diagonal. The assumption that the stabilizers are affine is reasonable because as noted by Totaro, the resolution property for stacks with non-affine stabilizer groups is not interesting \cite[Remark (1) on p. 2]{Tot}. One could interpret the above theorem as saying that a stack with the resolution property is very close to being separated.
Although there are simple examples of schemes which admit the resolution property but are not separated (see Example \ref{DVR with double point}), a stack with the $1$-resolution property is necessarily separated (see Lemma \ref{L1}).\\


To prove \ref{T}(1), we first reduce to the scheme case by using the fact that every normal algebraic space $X$ is a quotient of a normal scheme $Y$ by a finite group \cite[16.6.2]{LMB}. The excellence hypothesis is inserted to ensure that the regular locus of this scheme is open. Note that this also means that $X$ has finite normalization. Next we prove the theorem, first in the case where $X$ is a regular scheme and next in the case where $X$ is a normal scheme whose singular locus is contained in an affine open. The general case is then deduced from this by using Zariski's main theorem \cite[8.12.6]{DG4}.

To obtain \ref{T}(2), we first establish that such a stack is separated. This ensures that it has a coarse moduli space \cite[1.3]{KeelMori},\cite[6.12]{RydhQuotient}. We then use the slicing strategy of Kresch-Vistoli in \cite{KreschVistoli} to produce a finite flat covering of $\sX$.\\

In characteristic zero, we can improve the result further by using approximation techniques \cite{RydhApproximation} and drop all the finiteness hypotheses of Theorem \ref{T}. More precisely, we have the following version of Theorem \ref{T}.

\begin{theorem}\label{char0} Let $\sX/\Spec \mathbb{Q}$ be a qcqs integral normal algebraic stack whose stabilizers at closed points are affine. Then we have:
\begin{enumerate}
\item If $\sX$ is representable by an algebraic space $X$, then $X$ has the 1-resolution property if and only if it is quasi-affine.
\item $\sX$ has the 1-resolution property if and only if there exists a finite flat covering $U\rightarrow \sX$ with $U$ a quasi-affine scheme.
\end{enumerate}
\end{theorem}

This essentially reduces to showing that a special vector bundle remains special after approximation. We achieve this by using the fact that any tensor generator is also a strong tensor generator in characteristic zero (see \cite[Cor 6.6]{GrossRes}). This fails to hold in char $p$, as $\text{GL}_n$ is not linearly reductive.


We note at the outset that all our algebraic stacks and algebraic spaces will be quasi-separated. More precisely, they will have separated, quasi-compact diagonals (see \cite[D\'ef. 4.1]{LMB}). 
\\\\

\noindent{\bf Acknowledgements.} The first author was supported by the INSPIRE fellowship of the Department of Science and Technology, Govt.\ of India during the course of this work. We would like to thank Jack Hall for a number of helpful conversations, the idea of the proof of \ref{char0approx} is due to him. We are also grateful to David Rydh for his many helpful comments and suggestions. The proof of the non-noetherian version of Lemma \ref{separated} is due to him. We would also like to thank the referees for carefully reading an earlier version of this paper and for their helpful comments and corrections. 

\section{Examples and Counterexamples}
Before we present our main results we would like to mention a few examples and counterexamples of the 1-resolution property.

\begin{example}
If $G$ is a finite group over $k$ then we claim that $BG$ has the 1-resolution property. This follows from the fact that the map $\Spec(k)\rightarrow BG$ is a finite, faithfully flat cover and from \cite[Prop. 2.13]{GrossRes}.
\end{example}

\begin{example}\label{toruscase} Consider the multiplicative group $\G_m$ over a field $k$. The classifying stack $B\G_m$ does not have the 1-resolution property.

Note that any representation $V$ of $\G_m$ decomposes as a direct sum, $V\simeq \bigoplus_d \chi^{\oplus n_d}_d$, where $\chi_d$ are irreducible representations given by 

\begin{eqnarray*}
\chi_d :\G_m&\rightarrow&\G_m\\
t&\mapsto& t^d
\end{eqnarray*}

Assume (if possible) that $B\G_m$ has the 1-resolution property. Let $V$ be a special vector bundle on $B\G_m$. This means that $V$ is a $\G_m$-representation with the property that for any representation $W$ of $\G_m$, one can find a surjective map,
\begin{equation}\label{surj}
V^{\oplus m}\rightarrow W
\end{equation}
of $\G_m$-representations.
But since $V$ decomposes as a direct sum, $V\simeq \bigoplus_d \chi^{\oplus n_d}_d$, for any $\chi_r$ not appearing in the decomposition of $V$, equation \eqref{surj} fails. This contradicts our initial assumption.

\end{example}

\begin{example}Let $\G_a$ be the additive group over a field $k$.
We will show that $B\G_a$ does not have the 1-resolution property. \label{G_a}
The key idea is to identify an invariant $\eta(V)\in \N$, for every finite dimensional representation $V$ of $\G_a$ such that 
\begin{enumerate}
\item $\eta(V)= \eta(V^{\oplus r})\ \forall \ r\in \N$.
\item For all surjections $V\twoheadrightarrow W$, $\eta(W)\leq \eta(V)$. 
\item For every sufficiently large integer $m$, there exists a $V$ with $\eta(V)=m$.  
\end{enumerate}

We claim that the very existence of an invariant $\eta$ as above shows that $B\G_a$ cannot have the $1$-resolution property. Since otherwise, there would exist a vector bundle $V$ on $B\G_a$ such that every vector bundle $W$ is a quotient of $V^{\oplus r}$ for some $r\in \N$. This would gives us, by $(1)$ and $(2)$ above, that $\eta(W)\leq \eta (V)$ for all $W$, which contradicts $(3)$.

{\sc Case 1:} $char(k)=0$. In this case, finite dimensional representations of $\G_a$ are precisely pairs $(V,\rho)$, where $V$ is a finite dimensional $k$-vector space and $\rho$ is a nilpotent endomorphism of $V$ (see \cite[Example 4.9]{Abe} or \cite[pg 218, Example 2]{DemGab}). The map $\G_a\rightarrow \text{GL}(V)$ is given by $t\mapsto {\mathsf{exp}}(\rho t)$.
Thus any such representation $(V,\rho)$ admits a filtration by subrepresentations,
\[V\supseteq\rho(V)\supseteq\rho^2(V)\ldots\supseteq\rho^m(V)=\lbrace 0\rbrace,\]
with the property that the successive quotients are trival representations. We denote by $l(V)$ the smallest integer $m$ such that $\rho^m=0$. We claim that this invariant $l$ satisfies the above three properties. Given two representations, $(V,\rho)$ and $(W,\xi)$, we see that $l(V\oplus W)=max\lbrace \rho,\xi\rbrace$. Thus, $l$ satisfies conditions (1). That condition (2) is satisfied is straightforward. Moreover, by taking $\rho$ to be the $n\times n$ matrix with $1$'s on the superdiagonal (and all other entries zero), it is easy to construct $V$ such that $l(V)=n$ for every $n\in \N$. This shows condition (3).

{\sc Case 2:} $char(k)=p>0$. In this case, it is well known (see \cite[Example 4.9]{Abe} or \cite[pg 218, Example 2]{DemGab}) that a finite dimensional $\G_a$-representation $V$ is given by a sequence $\lbrace s_i\rbrace_{i\geq 0}$ of endomorphisms of $V$ such that, $s_i\circ s_j = s_j\circ s_i$, $s_i^p=0$, and $s_i=0$ for $i\gg 0$. The map $\G_a \to \text{GL}(V)$ is then given by 
$$t\mapsto \prod_i {\mathsf{exp}}(s_i t^{p^i})$$ where 
\[{\mathsf{exp}}(A):=1 + A + \frac{A^2}{2!}\ldots + \frac{A^{(p-1)}}{(p-1)!}.\]
We define our invariant to be the largest integer $n$ such that $s_n$ is non-zero, and denote it by $\gamma(V)$. Clearly, $\gamma(V\oplus W)=max\lbrace \rho,\xi\rbrace$ showing that condition (1) is satisfied. To see condition (3) for any $n\in \N$, take any non-zero nilpotent endomorphism $s$ of $V$, and consider the representation given by the sequence where $s_i=s \ \forall \ i\leq n$ and zero otherwise. Clearly $\gamma(V)=n$. We only need to show that condition (2) is satisfied. This is due to the following lemma:

\begin{lemma}
Let $V$ be a representation of $\G_a$ and $W\subseteq V$ a $k$-subspace. Then $W$ is a subrepresentation if and only if $s_i(W)\subseteq W$ for each $i$.
\end{lemma}
\begin{proof}
Let $\lbrace s_i\rbrace_{i\geq 0}$ be a representation of $V$. Let $W\subseteq V$ be a subrepresentation, and $\lbrace r_i\rbrace_{i\geq 0}$ be the corresponding endomorphisms of $W$. Then we have commutative diagram of comodule maps,

\begin{center}
\begin{tikzcd}
W\arrow[r]\arrow[d,"i"] & W\otimes k[t]\arrow[d,"i"]\\
V\arrow[r] & V\otimes k[t]
\end{tikzcd}
\end{center}
Comparing coefficients of powers of $t$, we see that $s_i|_W=r_i$. So, the $s_i$'s restrict to endomorphisms of $W$.
\end{proof}

\end{example}

\begin{example}\label{DVR with double point}
A DVR with a doubled point does not have the 1-resolution property. Let $R$ be a discrete valuation ring with function field $K$ and valuation $\nu$. Let $Y$ denote the scheme obtained by gluing two copies of $\Spec(R)$ along $\Spec(K)$. We will call this a {\it DVR with a doubled point}.

We will denote by $V_1$ and $V_2$ the two copies of $\Spec(R)$ in $Y$, and let $U:=\Spec(K)$ be the open point. Note that $U=V_1\cap V_2$.

Although we will show that $Y$ does not have the 1-resolution property, it is easy to see that every coherent sheaf on $Y$ is the quotient of some vector bundle, i.e, it has the resolution property. If $\sF$ is a coherent sheaf on $Y$, then on each of the $V_i's$, $\sF_{|V_i}:=\sF_i$ is just a finitely generated module over a PID. So, it decomposes into a free part and a torsion part. We have, $\sF_i\simeq \sF_{i,free}\oplus \sF_{i,tor}$, where $\sF_{i,free}$ and $\sF_{i,tor}$ are finitely generated free and torsion modules over $R$ respectively. Further, restricting $\sF_i$ to $U$ gives us the transition map,
\[A:\sF_1|_U\rightarrow\sF_2|_U\]
where $A$ is a $K$-linear map. Since the torsion part of $\sF_i$ vanishes on restriction to $U$, the transition map $A$ is completely determined by the free part of $\sF_i$. Hence, to produce a surjection from a vector bundle to $\sF$, we take an $n$ large enough so that on each $V_i$, there is a surjection,
$R^n\rightarrow \sF_{i,tor}$ onto the torsion part of $\sF_i$. This gives a map,
\[\sF_{i,free}\oplus R^n\rightarrow \sF_{i,free}\oplus\sF_{i,tor},\]
for each i. Glueing these maps on $U$ gives the required surjection. Thus, $Y$ has the resolution property.

To see that it does not have the 1-resolution property, assume  the contrary and let $\sE$ be a special vector bundle. Any vector bundle on $Y$ can be described by a pair $(n,A)$, where $n$ is the rank and $A\in \text{GL}_n(K)$ is the gluing map. In coordinates, let $A=(a_{ij})\in \text{GL}_n(K)$ be the gluing map $A: K^n\rightarrow K^n$ of $\sE$ on $U$. Let $\alpha:=min\lbrace \nu(a_{ij})\rbrace$ be the minimum of the valuations of the entries of $A$.  Consider a line bundle $\sL_{\lambda}$ such that the gluing map for $\sL_{\lambda}$ is given by $\lambda\in K^{\times}$ with $\nu(\lambda)<\alpha$.

As $\sE$ is a special vector bundle, there exists a surjection $\phi:\sE^{\oplus m}\rightarrow \sL_{\lambda}$, for some $m$. Restricting $\phi$ on each $V_i$, gives us the maps (written in coordinates as) $\phi_i:R^{N}\rightarrow R$, with $N:=nm$. Let $e_i$ be the chosen basis of $R^N$, and let $r_i:=\phi_1(e_i)$ and $s_i:=\phi_2(e_i)$. Note that $\nu(\phi_1(e_i))\geq 0$ for every $i$. After reordering the basis elements we may thus assume that $\nu(\phi_1(e_1)) = 0$. Otherwise, for any linear combination $\sum^n_{i=1}\alpha_i e_i$, the valuation of its image along $\phi_1$ would be
 $$\nu(\sum^n_{i=1}\alpha_i \phi_1(e_i))\geq min\lbrace \nu(\alpha_i e_i)\rbrace>0.$$ 
In this situation, the image would only generate an ideal of $R$, contradicting surjectivity.

Since $\phi$ is a homomorphism of sheaves, restricting $\phi_1$ and $\phi_2$ to $U$ gives the following commutative diagram,

\begin{center}
\begin{tikzcd}
K^{N}\arrow[r,"\phi_1"]\arrow[d,"A^{\oplus m}"]	&K\arrow[d,"\lambda"]\\
K^{N}\arrow[r,"\phi_2"]	&K
\end{tikzcd}
\end{center}

Note that along these maps, the chosen basis $e_i$ of $R^N$ restricts to a basis of $K^N$. So, by the above diagram we must have,
\[\lambda\phi_1(e_1)=\phi_2 (A^{\oplus m}(e_1)),\] 
or in coordinates, $\lambda r_1= \sum^n_{i=1} s_i a_{i1}$. But as $\nu(s_i)\geq 0$, comparing valuations of the left and right hand-side tells us that,
\[\nu(\sum^n_{i=1} s_i a_{i1})\geq min\lbrace \nu(s_i a_{i1})\rbrace\geq min\lbrace \nu(a_{i1})\rbrace>\nu(\lambda).\]
This contradicts commutativity of the above diagram.
\end{example}


\section{1-Resolution for Schemes}

The goal of this section is to prove Theorem \ref{T}(1). We begin by recalling (see \cite[Tag 07R3]{stacks-project}) that a locally Noetherian scheme $S$ is said to satisfy $J$-2 if for any locally finite type $S$-scheme $Y$, the regular locus of $Y$ is open in $Y$. All fields, $\Z$, Noetherian complete local rings, or schemes locally of finite type over these rings, give examples of $J$-2 schemes. All excellent schemes satisfy $J$-2.

We first note the following simple lemma.\begin{lemma}\label{lem:globallygenerated}
Let $X$ be a quasi-compact scheme with the $1$-resolution property. Let $\sE$ be a special vector bundle. Then $X$ is quasi-affine if and only if $\sE$ is globally generated. 
\end{lemma}
\begin{proof}
A quasi-compact scheme is quasi-affine if and only if every quasi-coherent sheaf of finite-type on it is globally generated (see \cite[5.1.2]{DG2} and \cite[1.7.16]{DG4}). Thus, if $X$ is quasi-affine, then $\sE$ is globally generated. Conversely, if the special vector bundle $\sE$ is globally generated then by \cite[Tag 01BB]{stacks-project} there exists a surjection $\sO^{\oplus m} \to \sE$. Let $F$ is a quasi-coherent sheaf of finite-type then since $\sE$ is special there is a surjection $\sE^{\oplus n} \to F$. Thus we obtain a surjection 
\[(\sO^{\oplus m})^{\oplus n} \to \sE^{\oplus n} \to F\] 
as desired. \end{proof}

\begin{proposition} \label{tensor} Let $X$ be a scheme with a special vector bundle $E$. If $V$ is a vector bundle then $E \otimes V$ is also special. 
\end{proposition}

\begin{proof}
Recall that there is a natural surjection $V \otimes V^{\vee} \to \mathcal{O}_X$ given by the trace map. Tensoring this surjection by an arbitrary quasi-coherent sheaf $F$ of finite-type we obtain an exact sequence
\[V \otimes V^{\vee} \otimes F \to F \to 0\]

Now to show that $E \otimes V$ is special, consider $F \otimes V^{\vee}$. Since $E$ is special there is a surjection $E^{\oplus m} \to F \otimes V^{\vee}$ and after tensoring this surjection by $V$ we obtain an exact sequence
\[(E \otimes V)^{\oplus m} \to F \otimes V^{\vee} \otimes V \to 0\]
combining this with the previous exact sequence shows that $E \otimes V$ is special. 
\end{proof}

\begin{example}

\label{quasi-proj} Let $X$ be a quasi-compact scheme with an ample line bundle $L$. Then $X$ has the 1-resolution property if and only if $X$ is quasi-affine. If $V$ is a special vector bundle, then it follows from \cite[Tag 01Q3 (8)]{stacks-project} that there exists an $n_0, k>0$ and a surjective morphism $((L^{\vee})^{\otimes n_0})^{\oplus k} \rightarrow V$. Thus, $(L^{\vee})^{\otimes n_0}$ is special. Since this sheaf is invertible, Proposition \ref{tensor} implies that $\mathcal{O}_X$ is also special. In other words every quasi-coherent sheaf of finite-type is globally generated and thus, $X$ is quasi-affine by Lemma \ref{lem:globallygenerated}.

\end{example}

\begin{proposition}\label{ample family} Let $X$ be a quasi-compact scheme admitting an ample family of line bundles. If $X$ has the $1$-resolution property then it is a quasi-affine scheme. 
\end{proposition}

\begin{proof} 
Let $L_1,...,L_n$ be line bundles along with sections $s_i \in H^0(X,L_i)$ with the property that the $D(s_i)$ are affine schemes which jointly cover $X$, i.e. $\bigcup D(s_i)=X$. If $E$ is a special vector bundle on $X$, then for any integer $n_i$, the vector bundle $E_i=E \otimes L_i^{\otimes n_i}$ is also special by Proposition \ref{tensor}. By \cite[Theorem 7.22]{Torsten}, we may choose a $n_i$ so that a generating collection of sections of $E$ over $D(s_i)$ lifts to sections of $E_i$ over $X$ i.e. we have a map $\phi_i: \mathcal{O}_X^{\oplus m_i} \to E_i$ whose cokernel is supported on $V(s_i)$. Thus by taking direct sums over $i$ we obtain
\[\bigoplus \phi_i: \bigoplus_{i=1}^n (\mathcal{O}_X^{\oplus m_i}) \to \bigoplus_{i=1}^n E_i\]
Now we will show that any quasi-coherent sheaf of finite-type, $F$, is globally generated. Indeed, since each $E_i$ is special we have surjections $f_i: E_i^{\oplus k_i} \to F$. We may assume that these $k_i's$ are all the same integer $k$. Thus we have
\[(\bigoplus f_i) \circ (\bigoplus \phi_i)^{\oplus k}: (\bigoplus_{i=1}^n (\mathcal{O}_X^{\oplus m_i}))^{\oplus k} \to \bigoplus_{i=1}^n (E_i)^{\oplus k} \to F\]
To check that this is surjective, it suffices to restrict the morphism to the cover consisting of $D(s_i)$'s. On $D(s_i)$ surjectivity follows because $\phi_i$ and $f_i$ are each surjective there.
\end{proof}

\begin{example} Consider $Y=V(x) \cup V(y) \subset \mathbb{A}^2 \backslash \{0\}=X$ and let $Y \to \mathbb{A}^1 \backslash \{0\}=Z$ be the map which is the identity on $V(x)$ and on $V(y)$ it sends $a \mapsto \frac{1}{a}$. Let $W$ denote the non-normal, $2$-dimensional variety obtained by taking the pushout $W=X \amalg_{Y} Z$. One can check that $W$ is not quasi-affine (see \cite{
stacks-project} Tag 0271). We will show that it doesn't have the 1-resolution property by showing it has an ample family of line bundles. Observe that $V(x-y)$ and $V(x+y)$ are Cartier divisors on $W$ whose complements yield an affine cover. Indeed, both $V(x-y)$ and $V(x+y)$ lie in the locus where $X \to W$ is an isomorphism, thus they induce effective Cartier divisors on $W$. Moreover, their complements admits a finite cover by $X\backslash V(x-y)$ or $X \backslash V(x+y)$, each of which is affine. In conclusion, because $W$ is not quasi-affine and admits an ample family of line bundles it cannot have the $1$-resolution property.
\end{example}

\begin{q}\label{Pushout}
Let $X,Y,Z$ be quasi-affine schemes which are finite-type over an excellent Noetherian scheme. Suppose we have a closed immersion $Y \to X$ and a finite morphism $Y \to Z$, then does the resulting pushout $X \amalg_{Y} Z$ admit an ample family of line bundles?
\end{q}

\begin{remark} \label{nonnormal} A positive answer to the above question will settle Question \ref{q} for any algebraic space $W$ finite type over an excellent affine base. To see this, argue by Noetherian induction on $W$. We may also suppose that $W$ is reduced. Indeed, $W_{red}$ also has the $1$-resolution property and if it is quasi-affine then $W$ is as well. Denote the normalization of $W$ by $X$ and let the conductor subscheme in $W$ be denoted $Z$, pulling back the conductor along the normalization map $X \to W$ yields a pinching diagram i.e. $W \cong X \amalg_{Y} Z$ (see \cite[2.2.5]{CLO}). Since $X$ is normal and has the $1$-resolution property it is quasi-affine, by Theorem \ref{T}(1). Since $Y$ is closed in $X$, it is also quasi-affine. Further, $Z$ also has the $1$-resolution property, and so $Z$ is quasi-affine by the induction hypothesis. Then, $W$ will admit an ample family of line bundles. So, it must be quasi-affine by Proposition \ref{ample family}.
\end{remark}

We will first prove the Theorem \ref{T}(1) in the special case when $X$ is a regular scheme. Note that in this case we do not use the assumption that $X$ is an excellent scheme. 
\begin{corollary}\label{regular}
Let $X$ be a normal Noetherian $\mathbb{Q}$-factorial scheme (e.g. if $X$ is regular). If $X$ has the 1-resolution property, then $X$ is quasi-affine.
\end{corollary}
\begin{proof}
Note that $X$ has affine diagonal \cite[Prop. 1.3]{Tot}. Hence, it has an ample family of line bundles by \cite[1.3]{SchroerAmple} Now, use Proposition \ref{ample family}.
\end{proof}

\begin{lemma}\label{separated}
Let $X$ be a qcqs algebraic space. If $X$ has the 1-resolution property, then $X$ is separated.
\end{lemma}
\begin{proof}
We may assume that $X$ is a scheme. The algebraic space case follows immediately because $X$ admits a finite cover by a qcqs scheme (see \cite[Theorem B]{RydhApproximation}), the 1-resolution property ascends along affine morphisms and separatedness descends along finite surjections. 

If $X$ is a scheme which is not separated then by the valuative criterion, there exists a valuation ring $V$, and two distinct maps $f_1,f_2:\Spec(V)\rightarrow X$, lifting the diagram:

\begin{center}
\begin{tikzcd}
\Spec(K)\arrow[d]\arrow[r,"x"] &X\arrow[d]\\
\Spec(V)\arrow[ru, dotted, "f_1",shift right=1.5ex]\arrow[ru,dotted,"f_2"']\arrow[r] &S
\end{tikzcd}
\end{center}

\noindent where $\Spec(K)$ is the generic point of $\Spec(V)$.

Consider the pullback diagram
\begin{center}
\begin{tikzcd}
W \arrow[d]\arrow[r] &X\arrow[d]\\
\Spec(V) \arrow[r] &X \times X
\end{tikzcd}
\end{center}
\noindent and note that $W \to \Spec(V)$ is an immersion of finite-type. Since it contains $\Spec K$ it is also schematically-dominant and hence a quasi-compact open immersion. It follows that the complement $\Spec(V) \backslash W$ can be described as the vanishing of finitely many elements $s_1,...,s_n \in V$. Since $V$ is a valuation ring, the $s_i$ with the smallest valuation generates the ideal $(s_1,...,s_n)$ i.e. $\Spec (V) \backslash W$ is a Cartier divisor and $W=\Spec(V_s)$ is affine. 

Let $Y$ be the scheme obtained by gluing two copies of $\Spec(V)$ along $W$. Consider the map $f:Y\rightarrow X$, which restricts to $f_1$ and $f_2$ on each of the copies of $\Spec(V)$, respectively. We will show that $f$ is affine. Let $U$ be an affine open in $X$. If $f^{-1}(U)\subset \Spec(V)$ then it is a quasi-compact open and by the preceding discussion is affine.

So, we may write
\[f^{-1}(U)=\Spec(V_f) \amalg_W \Spec(V_g).\]

\noindent Let $p\in \Spec(V_f)\setminus W$ and $q\in \Spec(V_g)\setminus W$ be two points. Note that if such $p, q$ do not exist then $f^{-1}(U)$ is affine. Since $V$ is a valuation ring, its ideals form a totally ordered set. Thus, without loss of generality we may assume that $p$ is a specialization of $q$. Thus, $q$ also lies in $\Spec(V_f)$. Thus, we have a diagram,

\begin{center}
\begin{tikzcd}
\Spec(K)\arrow[d]\arrow[r,"x"] &U\arrow[d]\\
\Spec(V_q)\arrow[ru, dotted, "f_1",shift right=1.5ex]\arrow[ru,dotted,"f_2"']\arrow[r] &S
\end{tikzcd}
\end{center}

\noindent As $V_q$ is a valuation ring, this contradicts the valuative criterion. So, $f^{-1}(U)$ is indeed affine.

It follows that $Y$ has the 1-resolution property. But $Y$ also admits an ample family of line bundles (e.g take the line bundles corresponding to each copy of $\Spec(V) \backslash W$) so Lemma \ref{ample family} implies that $Y$ must be quasi-affine. This is absurd since $Y$ is not separated. It follows that $f_1=f_2$ and the result follows from the valuative criterion of separatedness.
\end{proof}

\begin{remark} If we restrict ourselves to Noetherian schemes of finite type, then the proof above can be simplified. Indeed, in that case we are able to use the Noetherian valuative criterion and assume that $V$ is a discrete valuation ring. In this setting, $W$ is simply the generic point $\Spec(K)$. Moreover, every proper open set of $Y$ is affine. And as $f(Y)$ cannot be contained in any affine open of $X$ (by the valuative criterion), $f$ is affine. This leads to a contradiction by Lemma \ref{ample family} (or by Example \ref{DVR with double point}).
\end{remark}

In the sequel, we will repeatedly rely on \cite[Prop. 2.8 (v)]{GrossRes}, and so we state it here for convenience.

\begin{proposition}\cite[Prop. 2.8 (v)]{GrossRes}\label{qf}
Let $X\rightarrow Y$ be a quasi-affine morphism of algebraic stacks. If Y has the 1-resolution property then so does X.
\end{proposition}

\begin{lemma}
Let $X$ be a Noetherian normal, J-2 scheme with the property that the singular locus $B$ of $X$ is contained in an affine open $U$. Then, if $X$ has the 1-resolution property, it is quasi-affine.
\end{lemma}
\begin{proof}
Let $X_{\rm reg}$ denote the regular locus of $X$. Since the singular locus $B$ is contained in the affine open $U$, we have $X=X_{\rm reg}\cup U$. $X_{\rm reg}$ is quasi-affine by Corollary \ref{regular}. Let $Z_U:=X\setminus U$. As $U$ is affine, $Z_U$ is of pure codimension 1\cite[Tag 0BCU]{stacks-project}. Moreover, $Z_U$ lies entirely in $X_{\rm reg}$, and hence supports an effective Cartier divisor $D_U$. 

\noindent \underline{Step 1}: Let $\sE$ be a special vector bundle on $X$. We claim that for all $m\gg 0$ there exists an integer $r>0$ and a surjection $$ \big( \sO_X(-mD_U)\oplus \sO_X\big)^{\oplus r} \twoheadrightarrow \sE.$$

\noindent Since $U$ is affine, $\sE_{|U}$ is globally generated. We can choose sections $s_1,\ldots,s_{r_1}\in\Gamma(U,\sE)$ which generate $\sE$ on $U$. Similarly, as $X_{\rm reg}$ is quasi-affine, we can choose sections $t_1,\ldots,t_{r_2}\in\Gamma(X_{\rm reg},\sE)$ which generate $\sE$ on $X_{\rm reg}$. We may assume, without loss of generality, that $r_1=r_2=r$. By, normality of $X$, the restriction map $\Gamma(X,\sE)\rightarrow \Gamma(X_{\rm reg},\sE)$ is a bijection. Thus the sections $t_1,...,t_r$ extend uniquely to sections over the whole of $X$ giving us a map
$$ \sO_X^{\oplus r}\to \sE$$ which is surjective on $X_{reg}$.

By \cite[Theorem 7.22]{Torsten}, there exists an integer $m_U$, such that for each $m>m_U$, the $(s_D)^ms_i$'s (where $s_D$ is the section defining $D_U$) lift to global sections of $\mathcal{E}\tensor \sO_X(mD_U)$ thus giving a map 
 $$ \sO_X(-mD_U)^{\oplus r}\to \sE$$ which is surjective on $U$. 
Taking the direct sum of the above maps, gives a surjective map 
 $$  \big( \sO_X(-mD_U) \oplus \sO_X \big)^{\oplus r} \to \sE.$$

\noindent \underline{Step 2}: In this step we will show that there exists a line bundle $\sL$ such that both $\sL\oplus \sO_X$ and $\sL^2 \oplus \sO_X$ are special vector bundles.

This is straightforward,  since we can choose $\sL:=\sO_X(-m_0D_U)$, for some $m_0>m_U$. Then, the above step gives us surjections,

$$ \big( \sL\oplus \sO_X\big)^{\oplus r} \twoheadrightarrow \sE$$
$$ \big( \sL^2\oplus \sO_X\big)^{\oplus r} \twoheadrightarrow \sE$$
as required.

\noindent \underline{Step 3}: We now claim that $\sL^2\oplus \sO_X$ is globally generated. This will finish the proof of the theorem, thanks to Lemma \ref{lem:globallygenerated}.
Since $\mathcal{L}\oplus \sO_X$ is special, we have a surjection $$\Phi:(\mathcal{L}\oplus \sO_X)^{\oplus n}\twoheadrightarrow \mathcal{L}^2.$$ 

Hence, assume, if possible, that $p\in X$ is a point in the base locus of $\mathcal{L}^2$. Since the base locus of $\sL$ contains that of $\sL^2$, $p$ is also in the base locus of $\mathcal{L}$. The above surjection has as its summands, the  maps,
\[\mathcal{L}\overset{\phi_i}{\rightarrow}\mathcal{L}^2,\;\; \sO_X\overset{\psi_i}{\rightarrow}\mathcal{L}^2.\]

As $p$ is in the base locus, for all $i$, $\psi_{i,p}:\frac{\sO_{X,p}}{\mathfrak{m_p}}\rightarrow \frac{\mathcal{L}_p^2}{\mathfrak{m}_p\mathcal{L}_p^2}$ is the zero map.

Thus, we have an $i_0$ such that $\phi_{i_0,p}:\frac{\mathcal{L}_p}{\mathfrak{m}_p\mathcal{L}_p}\rightarrow\frac{\mathcal{L}_p^2}{\mathfrak{m}_p\mathcal{L}_p^2}$ is an isomorphism. Untwisting the map $\phi_{i_0}$ by $\mathcal{L}$, gives us a section of $\mathcal{L}$ which does not vanish at $p$, contradicting the initial assumption.
\end{proof}

\begin{theorem}\label{normal}
Let $X$ be a Noetherian normal scheme which is J-2. If $X$ has the 1-resolution property, then $X$ is quasi-affine.
\end{theorem}
\begin{proof}
We will show that the canonical map
\[\gamma: X\rightarrow \Spec(\Gamma(X,\mathcal{O}_X)),\]
is quasi-finite and separated. This will finish the proof as by Zariski's Main Theorem\cite[8.12.6]{DG4.3}, $\gamma$ will be quasi-affine, thus showing that $X$ itself is quasi-affine. 

That $\gamma$ is separated is clear from Lemma \ref{separated}. 

To show that $\gamma$ is quasi-finite, we let $X_{\rm reg}$ and $B$ denote the regular and singular loci of $X$, respectively. Let $B\subset\overset{n}{\underset{i=1}{\bigcup}} U_i$ be a covering, with $U_i$ affine in $X$, and $Y_i:=X_{\rm reg}\cup U_i$. Then, each $Y_i$ satisfies the hypothesis of the previous lemma, and hence is quasi-affine. Moreover, as each $Y_i$ contains $X_{\rm reg}$ it must contain all the codimension $1$ points and so we have 
\[\Gamma(X_{\rm reg},\mathcal{O}_X)\simeq \Gamma(Y_i,\mathcal{O}_X)\simeq \Gamma(X,\mathcal{O}_X)\]
so that $\Spec(\Gamma(Y_i,\mathcal{O}_X))= \Spec(\Gamma(X,\mathcal{O}_X))$. Thus, the canonical map
\[\gamma: X\rightarrow \Spec(\Gamma(X,\mathcal{O}_X))\]
is an open immersion on each $Y_i$. Hence, $\gamma$ is quasi-finite. 
\end{proof}

\begin{proof}[\textbf{Proof of Theorem \ref{T}(1)}]
Any Noetherian normal algebraic space is the quotient of a scheme by a finite group  \cite[16.6.2]{LMB}. So, we have a finite map, $U\rightarrow X$ with $U$ Noetherian normal. Since $X$ is quasi-excellent, $U$ also satisfies J-2. By Proposition \ref{qf}, $U$ has the 1-resolution property. Hence, it is quasi-affine by Theorem \ref{normal}. Then we can conclude that $X$ is quasi-affine using \cite[Prop. 4.7]{RydhQuotient}. Alternatively, as $U$ is quasi-affine, it is well-known that $X=U/G$ is a scheme, hence Theorem \ref{normal} applies and $X$ is quasi-affine as well.
\end{proof}

As mentioned in the introduction, we can do even better if we assume that $X$ is defined in characteristic zero. Let us begin with a lemma.

\begin{lemma} \label{char0approx}

Let $\sX/\Spec \mathbb{Q}$ be a qcqs integral normal algebraic stack whose closed points have affine stabilizers. Suppose that $\sX$ satisfies the 1-resolution property, then there is an affine morphism $\sX \to \sX_{\alpha}$ where $\sX_{\alpha}$ is a normal algebraic stack with affine diagonal which is finite-type over $\Spec \mathbb{Q}$ and has the $1$-resolution property. If $\sX$ is an algebraic space then $\sX_{\alpha}$ can be chosen to also be an algebraic space. 

\end{lemma}

\begin{proof}

By \cite[Theorem D]{RydhApproximation} we may write $\sX=\lim_{\lambda} \sX_{\lambda}$ where each $\sX_{\lambda}$ is an algebraic stack (or an algebraic space if $\sX$ is) of finite-type over $\Spec \mathbb{Q}$ and all the maps $\sX \to \sX_{\lambda}$ are affine and schematically dominant (see \cite[7.3]{RydhApproximation}). Note that because $\sX$ satisfies the 1-resolution property, Theorem \ref{Grossaffine} implies $\sX$ has affine diagonal. Thus, we may suppose that $\sX_{\lambda}$ has affine diagonal by \cite[Theorem C]{RydhApproximation} for all sufficiently large $\lambda$. Let $E$ be a special vector bundle on $\sX$ of rank $n$. Note that since $E$ is a tensor generator for $\sX$, the frame bundle of $E$ is a quasi-affine scheme (see \cite[6.4]{GrossRes}), we may even suppose that $E$ is self-dual by replacing $E$ with $E \oplus E^{\vee}$.

Now descend $E$ to some vector bundle $E_{\alpha}$ on a $\sX_{\alpha}$ so that $E=E_{\alpha}|_{\sX}$. Let $\text{Isom}(\mathcal{O}_{\sX_{\alpha}}^{\oplus n},E_{\alpha})=I_{\alpha}$. Since the formation of $I_{\alpha}$ is compatible with base change on $\sX_{\alpha}$, \cite[Tag 07SF(2)]{stacks-project} shows that
\[\text{Isom}(\mathcal{O}_{\sX}^{\oplus n}, E)=\lim_{\beta \geq \alpha} \text{Isom}(\mathcal{O}_{\sX_{\beta}}^{\oplus n}, E_{\beta})=\lim_{\beta \geq \alpha} I_{\beta}\]

\noindent where $E_{\beta}=E_{\alpha}|_{\sX_{\beta}}$. Thus since the left hand side is quasi-affine, \cite[Tag 07SR, 01Z5]{stacks-project} tells us that there is an $\alpha'$ so that for every $\beta \geq \alpha'$, $I_{\beta}$ is quasi-affine. Thus we may assume that the frame bundle of $E_{\alpha}$ has quasi-affine total space. Thus, $E_{\alpha}$ is a tensor generator (see \cite[6.4]{GrossRes}). Moreover, since $E$ is self-dual, we may suppose that $E_{\alpha}$ is as well. Since we are in characteristic $0$, it follows that $E_{\alpha}$ is a strong tensor generator (see \cite[6.6]{GrossRes}). We will now show that $E_{\alpha}$ can be chosen to be special. 

Observe that there is an $m>0$ so that there is a surjection $f: E^{\oplus m} \to E \otimes E$ by the specialness of $E$. Thus, for a sufficiently large index $\alpha$, we get a surjection $f_{\alpha}: E_{\alpha}^{\oplus m} \to E_{\alpha} \otimes E_{\alpha}$. Since $E_{\alpha}$ is a strong tensor generator which is self-dual, to show that $E_{\alpha}$ is special it suffices to show that every tensor power of $E_{\alpha}$ is a quotient of $E_{\alpha}^{\oplus N}$ for some $N>0$. We already have a surjection $f_{\alpha} \otimes \text{id}: E_{\alpha}^{\oplus m} \otimes E_{\alpha} \to E_{\alpha} \otimes E_{\alpha} \otimes E_{\alpha}$. However, the domain can be written as $(E_{\alpha}\otimes E_{\alpha})^{\oplus m}$ and so there is a surjection 
\[(f_{\alpha} \otimes \text{id}) \circ f_{\alpha}^{\oplus m}: (E_{\alpha}^{\oplus m})^{\oplus m} \to (E_{\alpha} \otimes E_{\alpha})^{\oplus m}=E_{\alpha}^{\oplus m} \otimes E_{\alpha} \to E_{\alpha} \otimes E_{\alpha} \otimes E_{\alpha}\]

\noindent By induction, we see that any polynomial expression in $E_{\alpha}$ is the quotient of an $E_{\alpha}^{\oplus M}$, i.e. that $E_{\alpha}$ is special. 

Note that \cite[Tag 0BB4 (4)]{stacks-project} shows that the schematically dominant morphism $f_{\alpha}: \sX \to \sX_{\alpha}$ factors through the normalization of $\sX_{\alpha}$:
 \[\sX \to (\sX_{\alpha})^{\nu}\]
Observe that $(\sX_{\alpha})^{\nu} \to \sX_{\alpha}$ is affine and hence $E_{\alpha}$ restricts to a special vector bundle on $(\sX_{\alpha})^{\nu}$. Since $\sX \to (\sX_{\alpha})^{\nu}$ is an affine morphism, we may conclude. 
\end{proof} 

\begin{proof}[\textbf{Proof of Theorem \ref{char0}(1)}]

By Lemma \ref{char0approx} there is an affine morphism $f_{\alpha}: X \to X_{\alpha}$ where $X_{\alpha}$ is a normal algebraic space, finite type over $\Spec \mathbb{Q}$ which has the $1$-resolution property. By \ref{T}(1) $X_{\alpha}$ is a quasi-affine scheme. Since $f_{\alpha}$ is affine it follows that $X$ is quasi-affine as well. 
\end{proof}

\begin{remark}
The above proof breaks down in characteristic $p$, because in that case $E_{\alpha}$ is not automatically a strong tensor generator. However, if the structure group of $E_{\alpha}$ is linearly reductive then the proof above works.
\end{remark}

\section{Finite Flat Covers}

In this section we establish the existence of a schematic finite flat cover of a stack $\sY$ with the $1$-resolution property. The first step is to produce a projective and flat morphism $f: \sX \to \sY$ where $\sX$ is generically a scheme and then pass to the associated map on coarse moduli spaces $\bar{f}: X \to Y$. We repeatedly replace $X$ with hyperplane sections in such a way that the resulting $X$ factors through $\sY$ where $X \to \sY$ is finite and flat. In order to ensure this, we replace $X$ with a hyperplane section not containing any of the components of the fibers of $\bar{f}$. Towards this end, we begin with the following Bertini-type result.

\begin{proposition} \label{FiberBertini}

Fix a Noetherian ring $R$ of finite dimension and assume that $R$ is universally catenary and Jacobson (e.g. $R$ is the ring of integers in a number field $K$). Let $X$ and $Y$ be algebraic spaces of finite-type over $\Spec R$ which are equipped with a proper morphism $f: X \to Y$ with constant fiber dimension $r>0$. If $X \to \mathbb{P}^n_{Y}$ is an immersion then there is a integer $d_0>0$ so that for every $d \geq d_0$ there is a open subset $U_d \subset H^0(\mathbb{P}^n_{R}, \mathcal{O}(d))=\mathbb{A}^N_{R}$ mapping surjectively to $\Spec R$ with the following property: for every geometric point of this open subset $\Spec \bar{k} \to U_d$, the corresponding hypersurface $H \subset \mathbb{P}^n_{\bar{k}}$ yields a morphism $X_{\bar{k}} \cap (H \times_{\bar{k}} Y_{\bar{k}}) \to Y_{\bar{k}}$ with constant fiber dimension $r-1$.

\end{proposition}

\begin{proof}

We begin by reducing to the case where $f$ is flat with geometrically irreducible fibers. In the course of the proof, we will fix the base scheme $\Spec R$ but vary the morphism $f$.  Let $P_f$ be the statement of the theorem for the morphism $f: X \to Y$. Moreover, for any finite-type morphism $Y' \to Y$ let $P_{f': X' \to Y'}$ (or simply, $P_{f'}$) be the statement that the theorem holds for $f'$ where $X'=X \times_Y Y'$. It follows that $P_f$ implies $P_{f'}$ and if $Y' \to Y$ is surjective then $P_{f'}$ implies $P_{f}$.

Observe that by Noetherian induction on $Y$ it suffices to show $P_{f'}$ is true for a nonempty open subscheme $Y' \subset Y$. Indeed, by Noetherian induction $P_{g}$ holds where $g$ is the base change of $f$ along $i: (Y \backslash Y')_{red} \subset Y$. Thus, by taking the larger of the $d_0$'s and intersecting the $U_d$'s (here the $d_0$ and $U_d$ are as in the statement of $P_{g}$ and $P_{f'}$), we may conclude that $P_f$ holds. Note that we may replace $Y$ with its reduction as $P_{f}$ is a topological statement. By shrinking $Y$ we may also assume it is integral, regular and affine. 

By shrinking $Y$ and replacing it with a connected finite \'etale cover if necessary (see \cite[Tag 020J]{stacks-project}) we may assume that the irreducible components, $\{F_i\}_{1 \leq i \leq n}$, of the generic fiber of $X \to Y$ are geometrically irreducible and that $\{\overline{F_i}\}_{1 \leq i \leq n}$ are the irreducible components of $X$. Denote by $f_i: \overline{F_i} \to Y$ the induced maps where each $\overline{F_i}$ inherits the reduced scheme structure. Shrinking $Y$ further, we may suppose that each $f_i$ is flat (see \cite[Tag 052A]{stacks-project}).

To show $P_{f}$ it suffices to show $P_{f_i}$ holds for the maps $f_i: F_i \to Y$ with constant fiber dimension $r_i>0$. Finally, since $f_i$ is flat with geometrically irreducible generic fiber we may shrink $Y$ further (see \cite[Tag 0559]{stacks-project}) so that the fibers of each $f_i$ are irreducible. Thus, we have reduced the proposition to the case where $X \to Y$ is a proper, flat morphism with irreducible fibers of constant positive dimension. 

Consider the exact sequence of sheaves that defines $X$ as a closed subscheme of $\mathbb{P}^n_Y$:
\[0 \to I_X \to \mathcal{O}_{\mathbb{P}^n_Y} \to \mathcal{O}_X \to 0\]

\noindent By \cite[2.2.1]{EGA3.2} there is a $l_0$ with the property that for every $d \geq l_0$ and $i>0$:
\[R^if_*(I_X(d))=R^if_*(\mathcal{O}_{\mathbb{P}^n_Y}(d))=R^if_*(\mathcal{O}_X(d))=0\]

\noindent and therefore
\[0 \to f_*I_X (d) \to f_* \mathcal{O}_{\mathbb{P}^n_Y} (d) \to f_*\mathcal{O}_X (d) \to 0\]

\noindent is an exact sequence of vector bundles whose formation is compatible with arbitrary base change on $Y$. Observe that the middle sheaf is a trivial vector bundle. We can write the total spaces over $Y$ of the first two vector bundles above as
\[i: V \to H^0(\mathbb{P}^n_{R}, \mathcal{O}(d)) \times_{\Spec R} Y \]
This inclusion can be described at any fiber $y \in Y$ as the subspace of those degree $d$ forms (with coefficients in the residue field of $y$) which vanish on $X_{y} \subset \mathbb{P}^n_{y}$. Thus, it suffices to show that the image of the composition
\[p_1 \circ i: V \to H^0(\mathbb{P}^n_{R}, \mathcal{O}(d)) \times_{\Spec R} Y \to H^0(\mathbb{P}^n_{R}, \mathcal{O}(d))\]

\noindent is disjoint from a dense open subset $U \subset H^0(\mathbb{P}^n_{R}, \mathcal{O}(d))=\mathbb{A}^N_R$ which surjects onto $\Spec R$. Since the fibers of $f$ have positive dimension we know that $p_{X_{y}}(d)$ eventually becomes larger than the dimension of $Y$. When this is true, we have
\[\text{dim} V = \text{dim}Y+p_{\mathbb{P}_y^n}(d)-p_{X_{y}}(d) < \text{dim} Y + p_{\mathbb{P}_y^n}(d)-\text{dim} Y=h^0(\mathbb{P}_y^n, \mathcal{O}(d))=N\]
This shows that $p_1 \circ i$ must have an image contained in a proper closed subset of $H^0(\mathbb{P}^n_{R}, \mathcal{O}(d))=\mathbb{A}^N_{R}$. Moreover, by \cite[Tag 0DS6]{stacks-project} applied to $p_1 \circ i$, the scheme-theoretic image of $V$ has dimension $< N$. As such, the complement of this image must surject onto $\Spec R$. Thus, $V$ is an open subset of $H^0(\mathbb{P}^n_{R}, \mathcal{O}(d))$ as in the proposition. \end{proof}

\begin{lemma} \label{avoidance} Fix a Noetherian, integral, affine scheme $\Spec R$ and suppose $C \subset \mathbb{P}^n_R$ is a proper closed subscheme which is nowhere dense in every fiber of $\mathbb{P}^n_R \to \Spec R$. Let $S$ denote the (finite) set of generic points of $C$. Then for any $d>0$ there is an open subset $V \subset H^0(\mathbb{P}^n_R, \mathcal{O}(d))=\mathbb{A}^N_R$ surjective over $\Spec R$ with the following property: for any finite flat cover $f: \Spec R' \to \Spec R$ and a $\Spec R'$-point of $V$, the induced hypersurface $H \subset \mathbb{P}^n_{R'}$ misses $S_{R'}=g^{-1}(S)$ where $g:\mathbb{P}^n_{R'} \to \mathbb{P}^n_{R}$.
\end{lemma}

\begin{proof} By intersecting an open set for each irreducible component of $C$ we may assume that $C$ is irreducible. Let $p: \Spec k(p) \to C$ be a closed point of $C$ and let $k(q)$ be the residue field of the image of $p$ in $\Spec R$. Then the kernel, $K$, of the natural morphism $H^0(\mathbb{P}^n_{k(q)}, \mathcal{O}(d)) \to k(p)$ is a proper closed subspace and corresponds to the collection of degree $d$ hypersurfaces passing through $p$. Thus 
\[K \subset H^0(\mathbb{P}^n_{k(q)}, \mathcal{O}(d))=\mathbb{A}^N_{k(q)} \subset H^0(\mathbb{P}^n_{R}, \mathcal{O}(d))=\mathbb{A}^N_R\]
\noindent is a closed subset not dense in any of the fibers of $\mathbb{A}^N_R \to \Spec R$. Let $V$ denote its complement. If $H$ is a hypersurface in $\mathbb{P}^n_{R'}$ induced by a $\Spec R'$-point of $V$ then we claim that $H$ cannot contain any generic point of $C_{R'}$. Indeed, suppose $H$ contains a generic point of $C_{R'}$, then it must contain the closure of that point. But since $\Spec R' \to \Spec R$ is finite this generic point must specialize to a point in $C_{R'}$ lying over $p$. In particular, $H$ must pass through this point. However, this violates the fact that $V$ misses $K$.  \end{proof}

\begin{proposition}
Let $R$ and $f:X \to Y$ be as in the statement of \ref{FiberBertini} where we also assume that $R$ is the ring of integers in a number field $K$ and the embedding $X \to \mathbb{P}^n_Y$ is induced by an embedding $X \to \mathbb{P}^n_R$. Suppose $C \subset X$ is a closed subscheme which is nowhere dense in any fiber over $\Spec R$. Then there is a finite flat morphism $\Spec B \to \Spec R$ (where $B$ is the ring of integers in a number field $L$) and a section $s \in H^0(\mathbb{P}^n_B, \mathcal{O}(d))$ whose vanishing, $H$, doesn't contain any component of $C_B$ and where $H \subset \mathbb{P}^n_{B}$ intersects $X_B \subset \mathbb{P}^n_{B}$ in such a way that
\[X_B \cap H \to Y_B\]
has constant fiber dimension $r-1$.
\end{proposition}

\begin{proof}
Consider $U_{d_0}$ as in the statement of Proposition \ref{FiberBertini} for the immersion $X \subset \mathbb{P}^n_Y$ induced by $i$. Next, choose an open subset $V_{d_0} \subset H^0(\mathbb{P}^n_{\Spec R}, \mathcal{O}(d_0))$ as in Lemma \ref{avoidance}. It suffices to show that the composition 
\[U_d \cap V_d \subset H^0(\mathbb{P}^n_{\Spec R}, \mathcal{O}(d))=\mathbb{A}^N_{\Spec R} \to \Spec R\]

\noindent admits a $\Spec B$-point where $B$ is a ring of integers in a number field $L$. However, as $R$ is the ring of integers in a number field $K$ we may apply \cite[1.6]{MBSkolem}. In this way we obtain a $\Spec \overline{R}$-point of $U_d \cap V_d$ where $\overline{R}$ denotes the integral closure of $R$ in $\overline{K}$. However, since $U_d \cap V_d$ is locally of finite presentation and $\overline{R}$ is the colimit of rings of integers, we can find a $\Spec B$-point as desired. 
\end{proof}

Recall that a stack $\sX$ is said to be a \emph{global quotient stack} if we can write $\sX=[Z/\text{GL}_n]$ where $Z$ is an algebraic space equipped with an action of $\text{GL}_n$.

\begin{thm}\label{finite flat}
Let $\mathcal{X}$ be an algebraic stack with finite diagonal which is finite-type over $\Spec \mathbb{Z}$ and denote by $X$ the corresponding coarse moduli space. Suppose that $\mathcal{X}$ is a global quotient stack and that $X$ admits an ample line bundle. Then $\mathcal{X}$ admits a finite flat cover $Z \to \mathcal{X}$ where $Z$ is a scheme with an ample line bundle. 
\end{thm}

\begin{proof} We follow the slicing strategy of \cite{KreschVistoli} with one small difference: we must allow ourselves to iteratively pass to finite flat covers of $\Spec \mathbb{Z}$. Indeed, let $V$ be a faithful vector bundle and consider $\mathcal{P}=\mathbb{P}(V \oplus \mathcal{O}) \to \mathcal{X}$. This is a smooth projective morphism with constant fiber dimension $r>0$ and a dense open subscheme $U \subset \mathcal{P}$ surjecting onto $\mathcal{X}$. Since the morphism is representable, by taking products of $\mathcal{P}$ with itself over $\mathcal{X}$ we may suppose that the dimension of the stacky locus of $\mathcal{P}$ is strictly smaller than that of the fiber dimension $r>0$. 

Passing to coarse moduli spaces we obtain a proper morphism $\pi: P \to X$ with constant fiber dimension $r>0$. We also have an open subset $U \subset P$ corresponding to the representable subset of $\mathcal{P}$, set $C=P \backslash U$ and note that by construction $\dim C<r$. 

Observe that $P \to X$ is a projective morphism: $\mathcal{P} \to \mathcal{X}$ has a relatively ample line bundle $L$ and for some $n$ we have that $L^{\otimes n}$ descends to $P$ (see \cite{RydhMO}) and that it is ample relative to $P \to X$. Indeed by the properness of $P \to X$, one can check ampleness after a finite cover on $X$, and we do this after a finite cover on $\mathcal{X}$ where it holds because $L^{\otimes n}$ is relatively ample for $\mathcal{P} \to \mathcal{X}$. Thus, because $X$ admits an ample line bundle, it follows that $P$ does as well and so we may fix an embedding $P \subset \mathbb{P}^n_{\mathbb{Z}}$. Now we apply the previous proposition to obtain a finite flat morphism $\Spec B_1 \to \Spec \mathbb{Z}$ and a hypersurface $H$ in $\mathbb{P}^n_{B_1}$ with the property that 
\[\pi_{B_1}: H \cap P_{B_1} \to X_{B_1}\]
has constant fiber dimension $r-1$ and $H$ doesn't contain any component of $C_{B_1}$. Observe that when restricted to $U_{B_1} \subset P_{B_1}$, $\pi_{B_1}$ factors through $\pi_{B_1}': H \cap U_{B_1} \to \mathcal{X}_{B_1}$ and $\pi_{B_1}'$ is flat with Cohen--Macaulay fibers. Indeed, flatness follows from the local criterion of flatness and the Cohen--Macaulay property follows because it is locally a complete intersection. Repeating this process $r-1$ more times, we obtain a finite flat cover $Z \to \mathcal{X}_{B_r}$. Indeed, for dimension reasons the $r$th slice must miss (the preimage of) $C$ entirely and so $\pi_{B_r}$ automatically factors through $\mathcal{X}_{B_r}$. By composing we obtain the finite flat cover $Z \to \mathcal{X}_{B_r} \to \mathcal{X}$ as desired. \end{proof}

\begin{corollary} \label{finiteflatgeneral}

Let $\mathcal{X}$ be a quasi-compact algebraic stack with finite diagonal and denote by $X$ the corresponding coarse moduli space (see \cite{RydhQuotient}). Suppose that $\mathcal{X}$ is a global quotient stack and that $X$ admits an ample line bundle. Then $\mathcal{X}$ admits a finite flat cover $Z \to \mathcal{X}$ where $Z$ is a scheme with an ample line bundle. 

\end{corollary}

\begin{proof} 

By \cite[Theorem D]{RydhApproximation} we may write $\mathcal{X}=\lim_{\lambda} \sX_{\lambda}$ where each $\sX_{\lambda}$ is of finite type over $X$ and the morphisms $\sX \to \sX_{\lambda}$ are affine and schematically dominant (see \cite[7.3]{RydhApproximation}). We may also suppose that each $\sX_{\lambda}$ is a global quotient stack and is separated over $X$ (see \cite[Theorem C]{RydhApproximation}). Indeed, the faithful vector bundle on $\sX$ must descend to a faithful vector bundle on some $\sX_{\lambda}$ (see the proof of Lemma \ref{char0approx}). We may also suppose that each $\sX_{\lambda}$ also has finite diagonal so that it must itself have a coarse moduli space. It immediately follows from the universal properties of coarse moduli spaces that the natural map $\sX_{\lambda} \to X$ is a coarse moduli morphism. Thus, by replacing $\sX$ by $\sX_{\lambda}$ we may assume that $\mathcal{X} \to X$ is a coarse moduli morphism which is of finite-type and thus which is proper and quasi-finite (see \cite[6.12]{RydhQuotient}). Now apply \cite[7.4]{RydhApproximation} to the finite type morphism $\sX \to X$ so we may write $\sX=\lim_{\lambda} \sX_{\lambda}$ where the bonding maps are closed immersions and each $\sX_{\lambda} \to X$ is of finite-presentation. By \cite[6.4, 6.5]{RydhApproximation} we may further assume that each $\sX_{\lambda}$ has finite diagonal and is quasi-finite over $X$. Thus, by replacing $\sX$ with one such $\sX_{\lambda}$ we may assume that $\sX$ is of finite-presentation and quasi-finite over a scheme $X$ with an ample line bundle. Note that now, $X$ no longer denotes the coarse moduli space of $\sX$.

Next, we may write $X=\lim_{\alpha} X_{\alpha}$ where $X \to X_{\alpha}$ is affine and each $X_{\alpha}$ admits an ample line bundle and is finite-type over $\Spec \mathbb{Z}$. By \cite[Tag 0CN4]{stacks-project} (or \cite[2.2]{OlssonHom}) and \cite[6.5]{RydhApproximation} we may find a stack $\sX_{\alpha} \to X_{\alpha}$ which has a finite diagonal along with the property that $\sX=\sX_{\alpha} \times_{X_{\alpha}} X$. Now since $\sX \to X$ is a global quotient stack which is quasi-finite and separated over $X$, \cite[7.12]{RydhApproximation} allows us to find an $\alpha$ so that $\sX_{\alpha} \to X_{\alpha}$ enjoys the same properties. We immediately see that there is a natural factorization
\[\sX_{\alpha} \to (\sX_{\alpha})_{cms} \to X_{\alpha}\]
Observe that $(\sX_{\alpha})_{cms} \to X_{\alpha}$ is separated as well as quasi-finite and hence the coarse moduli space of $\sX_{\alpha}$ also admits an ample line bundle. Indeed, the pullback of an ample line bundle along a separated and quasi-finite map is ample. Thus, we may now apply Theorem \ref{finite flat} to deduce the existence of a finite flat $Z \to \sX_{\alpha}$ where $Z$ is quasi-projective. We conclude by base changing $Z \to \sX_{\alpha}$ along the affine morphism $\sX \to \sX_{\alpha}$. \end{proof}

\begin{remark} \label{rydh}
David Rydh has made us aware that he has a sketched an argument for Corollary \ref{finiteflatgeneral} which is stronger than the one we have given here. Among other things, he is able to control the degree and regularity locus of the finite flat cover. Moreover, he believes it is possible to remove the use of \cite{MBSkolem} using the techniques of \cite{GGL}. 
\end{remark} 

\section{The 1-Resolution property for Stacks}

We now prove the existence of finite flat covers for stacks with the 1-resolution property. To that end, we must first establish that such a stack is separated. We begin with the following special case.

\begin{proposition}\label{finite}
Let $G$ be an affine algebraic group over $k$. Then, $BG$ has the 1-resolution property if and only if $\dim(G)=0$.
\end{proposition}

\begin{proof}
($\Leftarrow$) If $\dim(G)=0$, the map $\Spec(k) \to BG$ is finite and faithfully flat. Thus, the $1$-resolution property for $BG$ follows from \cite[Prop. 2.13]{GrossRes}. 

($\Rightarrow$) We may assume that $k$ is algebraically closed and $G$ reduced and connected.  Let $T \subset G$ be a maximal torus. The following diagram is cartesian,

\begin{center}
\begin{tikzcd}
G/T\arrow[r]\arrow[d] &\Spec(k)\arrow[d]\\
BT\arrow[r] & BG
\end{tikzcd}
\end{center}
\noindent As the scheme $G/T$ is affine, it follows from fppf descent that the morphism  $BT\rightarrow BG$ is affine. Hence by Proposition \ref{qf}, $BT$ has the 1-resolution property. However this is impossible if $T$ is a positive dimensional torus (see Example \ref{toruscase}).  Thus $G$ must be unipotent. Assume $\dim(G)> 0$, if possible.  By using the derived series of $G$, one can find a closed subgroup of $G$ isomorphic to $\G_a$. The inclusion  
$ \G_a\inj G$ gives rise to a quasi-affine morphism $B\G_a \to BG$. Again, by Proposition \ref{qf}, we deduce that $B\G_a$ has the $1$-resolution property, which is impossible (see Example \ref{G_a}). Thus $\dim(G)$ must be $0$. 
\end{proof}

The above result implies that stacks with the 1-resolution property have finite isotropy groups.

\begin{proposition}\label{quasi-finite}
Let $\sX$ be an algebraic stack whose stabilizers at closed points are affine. If $\sX$ has the 1-resolution property, then the inertia stack $f:I_\sX \to \sX$ is quasi-finite. Equivalently, the diagonal $\Delta:\sX \to \sX \times \sX$ is quasi-finite.
\end{proposition}

\begin{proof} The inertia stack is of finite type since this is true of the diagonal $\Delta$ (see \cite[4.2]{LMB}). Thus, to show that $f$ is quasi-finite, it is enough to show that it has finite fibres. Let $\xi: \Spec(k) \to \sX$ be a point.
We need to show that the stabilizer group at $\xi$ is finite. Let $i: \sG_{\xi}\inj \sX$ be the residual gerbe at $\xi$. Then $i:\sG_{\xi}\rightarrow\sX$ is quasi-affine \cite[Theorem B.2]{Rydh11} and by Proposition \ref{qf}, $\sG_{\xi}$ has the $1$-resolution property. Now, finiteness of the stabilizer group at $\xi$ follows directly from Proposition \ref{finite}.
\end{proof}

The following is a refinement of \cite[Lemma 6.1]{Tot}.
\begin{lemma} \label{schemecover}
Let $\sX$ be an algebraic stack and $a:X\to \sX$ be an integral surjective map from a scheme $X$. Then if $X$ is separated then $\sX$ is as well.
\end{lemma}


\begin{proof}

Assume, if possible that $\sX$ is not separated. Then by the valuative criterion  there exists a valuation ring $R$, with two distinct maps $f_1,f_2:\Spec(R)\rightarrow \sX$ making the following diagram commute:

\begin{center}
\begin{tikzcd}
\Spec(K)\arrow[d]\arrow[r,"x"] &\sX\\
\Spec(R)\arrow[ru, dotted, "f_1",shift right=1.5ex]\arrow[ru,dotted,"f_2"']
\end{tikzcd}
\end{center}

\noindent Here, $K\supseteq R$ is the field of fractions of $R$. Without loss of generality, by going to an extension of $K$, we may assume that the map $x:\Spec(K) \to \sX$ lifts to give a map 
$\tilde{x}: \Spec(K) \to X$. Since the map $a:X\to \sX$ is integral and hence universally closed and affine, it satisfies the existence part of the valuative criterion (see \cite[Tag 0CLX]{stacks-project}). Thus, the maps $f_1,f_2$ also lift to give a commutative diagram (after going to an extension of $R$),

\begin{center}
\begin{tikzcd}
\Spec(K)\arrow[d]\arrow[r,"\tilde{x}"] & X\\
\Spec(R)\arrow[ru, dotted, "\tilde{f}_1",shift right=1.5ex]\arrow[ru,dotted,"\tilde{f}_2"']
\end{tikzcd}
\end{center}
However this contradicts the separatedness of $X$. Hence $\sX$ must be separated. 
\end{proof}

\begin{lemma}\label{L1}
Let $\sX$ be a qcqs algebraic stack whose stabilizer groups at closed points are affine. If $\sX$ has the 1-resolution property then $\sX$ is separated, and hence has finite diagonal.
\end{lemma}

\begin{proof} Since the diagonal $\Delta:\sX\rightarrow \sX\times\sX$ is quasi-finite by Proposition \ref{quasi-finite},   $\sX$ admits a finite surjective map $a:Z\rightarrow\sX$ where $Z$ is a scheme (see \cite[Theorem B]{RydhApproximation}). Then, by Proposition \ref{qf}, $Z$ also has the $1$-resolution property and hence is separated by Lemma \ref{separated}. This shows that $\sX$ is separated by Lemma \ref{schemecover}. 
\end{proof}

We now finish with the proofs of Theorems \ref{T} and \ref{char0}.

\begin{proof}[\textbf{Proof of Theorem \ref{T}(2) \& Theorem \ref{char0}(2)}] If $\sX$ lies over $\Spec \mathbb{Q}$ then we may use Lemma \ref{char0approx} to produce an affine morphism $\sX \to \sX_{\alpha}$ where $\sX_{\alpha}$ is a normal algebraic stack which has the $1$-resolution property and is finite-type over $\mathbb{Q}$. If we show the result for $\sX_{\alpha}$ then by pulling back the quasi-affine cover by the affine morphism $\sX \to \sX_{\alpha}$ we may conclude. Thus, it suffices to assume that $\sX$ is finite-type over a Noetherian excellent base.

Since the diagonal of $\sX$ is affine (see \cite[1.1]{GrossRes}) and quasi-finite, there is a finite surjective morphism $Z \to \sX$ (see \cite[Theorem B]{RydhApproximation}) where $Z$ is a scheme. After replacing $Z$ with its normalization we may apply \ref{char0}(1) or \ref{T}(1) to deduce that $Z$ is quasi-affine. Note that the diagonal of $\sX$ is proper by Lemma \ref{L1} and quasi-finite. This means $\sX$ admits a coarse moduli space $X$ (see\cite[Theorem 6.12]{RydhQuotient}) and since $\sX$ is normal so is $X$. Indeed, since coarse moduli spaces are geometric quotients, the normality of $X$ follows from the proof of \cite[Theorem 4.16(viii)]{goodmodulispaces}. The morphism $Z \to X$ is integral and surjective and since $Z$ admits an ample line bundle, \cite[3.5]{GrossRes} implies that $X$ must also have an ample line bundle. Since $\sX$ is a global quotient stack (see \cite[5.10]{GrossRes}) with quasi-projective coarse space, we obtain a finite flat morphism $Z' \to \sX$ where $Z'$ admits an ample line bundle (see Corollary \ref{finiteflatgeneral}). Since $Z'$ inherits the $1$-resolution property from $\sX$ and has an ample line bundle, it must be quasi-affine by Example \ref{quasi-proj}.
\end{proof}

\begin{remark} \label{normalstack} In the proof above, the normality of $\sX$ is necessary in order to descend an ample line bundle along the map $Z \to X$. Indeed, to apply Corollary \ref{finiteflatgeneral} to $\sX$ we require an ample line bundle on the coarse moduli space $X$. However, positive answers to Question \ref{Pushout} \emph{and} the following question would allow us to generalize Theorem \ref{T}(2) and Theorem \ref{char0}(2) to the non-normal setting. \end{remark}

\begin{q} \label{nonnormalstack} Let $\sX$ be an algebraic stack which is finite-type over an excellent base and suppose $\sX$ has the 1-resolution property. Does the coarse moduli space of $\sX$ have the 1-resolution property? \end{q}


\bibliography{mybib}{}
	\bibliographystyle{plain}

\end{document}